\documentclass[reqno]{amsart}

\usepackage[utf8]{inputenc}
\usepackage[english]{babel}
\usepackage{stmaryrd}
\usepackage{amsthm}
\usepackage{amsmath}
\usepackage{amssymb}
\usepackage{enumitem}
\usepackage[dvips]{graphicx}
\usepackage{color}
\usepackage{hyperref}
\usepackage{url}
\usepackage[left=3.5cm, right=3.5cm, paperheight=11.5in]{geometry}
\usepackage{fancyhdr}
\usepackage{mathrsfs}
\usepackage{bm}
\usepackage{stmaryrd}

\newtheorem{theorem}{Theorem}
\newtheorem{lemma}[theorem]{Lemma}
\let\oldlemma\lemma
\renewcommand{\lemma}{\oldlemma\normalfont}

\newtheorem{corollary}[theorem]{Corollary}
\let\oldcorollary\corollary
\renewcommand{\corollary}{\oldcorollary\normalfont}
\newtheorem{prop}[theorem]{Proposition}
\let\oldprop\prop
\renewcommand{\prop}{\oldprop\normalfont}
\newtheorem{rmk}[theorem]{Remark}
\let\oldrmk\rmk
\renewcommand{\rmk}{\oldrmk\normalfont}
\newtheorem{example}[theorem]{Example}
\let\oldexample\example
\renewcommand{\example}{\oldexample\normalfont}

\theoremstyle{definition}

\theoremstyle{remark}

\hypersetup{
    pdftitle={On consecutive perfect powers},
    pdfauthor={Paolo Leonetti},
    pdfmenubar=false,
    pdffitwindow=true,
    pdfstartview=FitH,
    colorlinks=true,
    linkcolor=red,
    citecolor=green,
    urlcolor=cyan
}

\renewcommand{\rho}{\varrho}

\begin{document}
\title{On consecutive perfect powers with elementary methods}  

\author{Paolo Leonetti}
\address{Department of Statistics, Universit\`a ``Luigi Bocconi'' -- via Roentgen 1, 20136 Milano, Italy}
\email{leonetti.paolo@gmail.com}
\urladdr{https://sites.google.com/site/leonettipaolo/}



\begin{abstract}
Catalan's conjecture claims that the Diophantine equation $x^p-y^q=1$ admits the unique solution $3^2-2^3=1$ in integers $x,y,p,q \ge 2$. The conjecture has been finally proved by P. Mih\u{a}ilescu (2002) using the theory of cyclotomic fields and Galois modules.

Here, relying only on elementary techniques, we prove several instances of this classical result. In particular, we prove the conjecture in the following cases: $p$ even (due to V.A. Lebesgue), $q$ is even (due to L. Euler and Chao Ko), $x$ divides $q$, $y$ divides $x-1$, $y$ is a power of a prime, and $y\le p^{p/2}$.
\end{abstract}
\maketitle
\thispagestyle{empty}


\section{Introduction}\label{sec:introduction}


In $1844$ the Belgian mathematician 
Eug\`{e}ne Charles Catalan \cite{Cat1} 
conjectured that $8$ and $9$ are the only (positive) consecutive perfect powers. That is to say, the unique solution of the Diophantine equation
\begin{equation}\label{eq:main}
x^p-y^q=1
\end{equation}
in integers $x,y,p,q \ge 2$ is given by $3^2-2^3=1$. In this respect, he only made some empirical observations, which he stated without proof: e.g., he considered the special cases $(x+1)^x-x^x=1$ and $x^y-y^x=1$.

The history of the problem dates back to at least Levi ben Gerson (also known as Gersonides), who considered in $1343$ the Diophantine equations 
$$
2^p-3^q=1 \,\, \text{ and }\,\,3^p-2^q=1.
$$

Some years after Catalan made his conjecture, V.A. Lebesgue \cite{Lebesgue} (not to be confused with his more famous namesake H.L. Lebesgue) solved the case $q=2$. In the early 1900's, L. Euler \cite{Eulero} and Chao Ko \cite{Ko} gave a solution of the conjecture in the case $p=2$ (it is worth noting that a particular instance of this result, i.e., the case $p=4$, has been previously shown by Selberg). Their solutions are surveyed in Sections \ref{sec:caseqeven} and \ref{sec:casepeven}, respectively. 

After years of study, R. Tijdeman \cite{Tij} made an important breakthrough in $1976$: by means of Baker's method in trascendence theory, he was able to prove that Catalan's equation \eqref{eq:main} has at most finitely many solutions. In addition, he shows the existence of an effective computable constant $C$ such that $p,q \le C$ whenever $x^p-y^q=1$ for some $x,y \ge 2$. Then, M. Langevin \cite{Langevin} computed the explicit upper bound
$$
C \le e^{e^{e^{e^{730}}}}.
$$
More recently, M. Mignotte \cite{Mignotte} improved this upper bound, showing that every solution of Catalan's equation satisfies
$$
p<7.15\cdot 10^{11}\,\,\text{ and }\,\,q<7.78\cdot 10^{16}.
$$
Unfortunately, it appeared that the required computation to complete the proof was too huge to be feasible (at this time, the conjecture could be proved by case checking only for $p,q<10^7$).

Finally, P. Mih\u{a}ilescu \cite{Miha} solved completely the conjecture in April $2002$, with a brilliant proof which relies on cyclotomic fields and Galois modules (the result was published two years later). In this regard, Catalan's conjecture is nowadays known as \emph{Mih\u{a}ilescu's theorem}. 

Here, we are not going to present Mih\u{a}ilescu's proof, which can be found, e.g., in Schoof's monograph \cite{Cat2}. On the other hand, the aim of this article is to provide and collect elementary proofs for several interesting cases of Catalan's equation. The arguments used can be followed by diligent students who have taken a first course in Number Theory.

First, a remark is in order: it can be assumed without loss of generality that $p$ and $q$ are distinct primes. Indeed, if $p^\prime$ and $q^\prime$ are primes dividing $p$ and $q$, respectively, then
$$
x^p-y^q =X^{p^\prime}-Y^{q^{\prime}},
$$
where $X:=x^{p/p^\prime}$ and $Y:=y^{q/q^\prime}$ are integers $\ge 2$. In addition, if $p=q$ then
\begin{equation}\label{eq:pequalq}
x^p-y^p \ge (y+1)^p-y^p \ge 1+py >1,
\end{equation}
i.e., there are no two consecutive $p$-th powers. With these premises, our main result follows.



\begin{theorem}\label{th:main}
Catalan's conjecture holds in each of the following cases:
\begin{enumerate}[label={\rm (\roman{*})}]
\item\label{item:(i)} $q=2$;
\item\label{item:(ii)} $p=2$;
\item\label{item:(iii)} $x$ is a power of $2$;
\item\label{item:(iv)} $x\equiv 3,5,7 \pmod{8}$;
\item\label{item:(v)} $x$ divides $q$;
\item\label{item:(vi)} $x\equiv 1 \pmod{y}$; 
\item\label{item:(vii)} $y$ is a power of a prime;
\item\label{item:(viii)} $y \le \min\left(\frac{(pq)^p}{q},\left(\frac{q\sqrt{p}}{2}\right)^p\right)$. 
\end{enumerate}
\end{theorem}

In particular, case \ref{item:(viii)} implies that Catalan's conjecture holds if $y\le p^{p/2}$. Proof of Theorem \ref{th:main} follows, case by case, in Sections \ref{sec:caseqeven}-\ref{sec:y5p}. Closing remarks and several open questions related to Catalan's equation are given in Section \ref{sec:final}.

\subsection{Notation} We let $\mathbf{N}$ and $\mathbf{N}^+$ represent the set non-negative integers and positive integers, respectively (in particular, $0 \in \mathbf{N}$). Moreover, given a prime $p$ and $n \in \mathbf{N}^+$, we let $\upsilon_p(n)$ be the $p$-adic valuation of $n$, i.e., the greatest $k \in \mathbf{N}$ such that $p^k$ divides $n$.



\subsection{Acknowledgements}The author is grateful to Salvatore \textsc{Tringali} (University of Graz, Austria) for suggesting remarks that improved the readability of the article.


\section{Proof of Theorem \ref{th:main}.\ref{item:(i)}: $q=2$.}\label{sec:caseqeven}

Here, we are going to show that the Diophantine equation
\begin{equation}\label{eq:q2}
y^2+1=x^p
\end{equation}
has no integral solutions $\ge 2$. 

The case $p=2$ does not have any solution, as it follows from inequality \eqref{eq:pequalq}. Hence, let us assume that $p\ge 3$ (here, we recall that there is no loss of generality to say that $p$ is a prime). 
Clearly, $x$ and $y$ have different parity. If $x$ is even then $y$ is odd and $y^2+1$ has to be divisible by $4$, which is turn impossible because
\begin{displaymath}
0\equiv x^p=y^2+1=(\underbrace{y+1}_{\mathrm{even}})(\underbrace{y-1}_{\mathrm{even}})+2 \equiv 2\pmod{4}.
\end{displaymath}
This implies that $x$ has to be odd and $y$ even. 

At this point, let us rewrite Equation \eqref{eq:q2} in $\mathbf{Z}[i]$ as
\begin{displaymath}
(y+i)(y-i)=x^p,
\end{displaymath}
where $i$ is the imaginary number for which, by definition, $i^2=-1$. Then, on the one hand, the greatest common divisor $\mathrm{gcd}(y+i,y-i)=\mathrm{gcd}(y+i,2i)$ divides $2i$. On the other hand, $i$ is unit in $\mathbf{Z}[i]$ and $x$ is odd. It follows that $y+i$ and $y-i$ have to be (coprime) $p$-th powers in $\mathbf{Z}[i]$. Therefore there exist integers $a,b$ for which
$$
y+i=(a+bi)^p=\sum_{j=0}^p{\binom{p}{j}a^j(bi)^{p-j}}. 
$$
In particular, the imaginary parts of the both sides are equal, i.e.,
\begin{equation}\label{qeven3} 
1=\sum_{j=0}^{\frac{p-1}{2}}{\binom{p}{2j}a^{2j}b^{p-2j}i^{p-2j-1}}=b\sum_{j=0}^{\frac{p-1}{2}}{\binom{p}{2j}a^{2j}b^{p-2j-1}(-1)^{\frac{p-1}{2}-j}}. 
\end{equation}
Since the right hand side is divisible by $b$, then $|b|=1$. Multiplying both sides by $(-1)^{\frac{p-1}{2}}b$, it follows that we can rewrite Equation \eqref{qeven3} as
\begin{equation}\label{eq:fkjhdgs} 
\sum_{j=0}^{\frac{p-1}{2}}{\binom{p}{2j}(-a^2)^{j}}=(-1)^{\frac{p-1}{2}}b. 
\end{equation}

Considering also that 
$$
y-i=\overline{y+i}=\overline{(a+bi)^p}=\left(\overline{a+bi}\right)^p=(a-bi)^p,
$$
we obtain that, for some odd $x$, it holds
$$
x^p=(a+bi)^p(a-bi)^p=(a+i)^p(a-i)^p=(a^2+1)^p.
$$
In particular, $a$ is even and different from $0$ (because $a^2+1=x\ge 2$). It follows that
$$
\sum_{j=0}^{\frac{p-1}{2}}{\binom{p}{2j}(-a^2)^{j}} \equiv 1\pmod 4,
$$
which implied, together with \eqref{eq:fkjhdgs}, that 
\begin{equation}\label{eq:finalp}
\sum_{j=0}^{\frac{p-1}{2}}{\binom{p}{2j}(-a^2)^{j}}=1.
\end{equation}

At this point, if $p=3$ then the above equation implies $a=0$, which is impossible. Otherwise $\frac{p-1}{2}\ge 2$, hence Equation \eqref{eq:finalp} can be further rewritten as
\begin{equation}\label{qeven4}
\sum_{j=2}^{\frac{p-1}{2}}{\binom{p}{2j}(-a^2)^{j}}=a^2\binom{p}{2}.
\end{equation}

Finally, it is claimed that the number of factors $2$ dividing each term in the left hand side of the above equation is strictly greater than the one on the right hand side. Indeed, considering that the identity
$$
\binom{p}{2j}=\frac{1}{j(2j-1)}\binom{p-2}{2j-2} \binom{p}{2}
$$
holds for all $j=2,\ldots,\frac{p-1}{2}$, we obtain that 
%
\begin{displaymath}
\begin{split}
\upsilon_2\left(a^{2j} \binom{p}{2j} \right) &=\upsilon_2\left(\frac{a^{2j}}{j(2j-1)}\binom{p-2}{2j-2} \binom{p}{2}\right) \\
&\ge 2j\upsilon_2(a)-\upsilon_2(j)+\upsilon_2\left(\binom{p}{2}\right) \\
&\ge \left(2j-\upsilon_2(j)\right)\upsilon_2(a)+\upsilon_2\left(\binom{p}{2}\right),
\end{split}
\end{displaymath}
which in turn is strictly greater than $\upsilon_2\left(a^2\binom{p}{2}\right)=2\upsilon_2(a)+\upsilon_2\left(\binom{p}{2}\right)$. 
This implies that Equation \eqref{qeven4} has no integral solutions, concluding the proof.


\section{Proof of Theorem \ref{th:main}.\ref{item:(ii)}: $p=2$.}\label{sec:casepeven}

Since two squares of positive integers cannot be consecutive, it is enough to check the cases $q=3$ and $q\ge 5$. 



\subsection{Case $q=3$.} In this case, Equation \eqref{eq:main} can be rewritten as
\begin{equation}\label{eq:yyycubes}
y^3=(x+1)(x-1).
\end{equation}
Notice that $x+1$ and $x-1$ are positive integers which differ by $2$, hence they cannot be both cubes. Since their greatest common divisor divides their difference, then $\mathrm{gcd}(x+1,x-1)$ is exactly $2$. In particular, $x$ is odd and $y$ is even: let us say $x=2m+1$ and $y=2n$ for some $m,n \in \mathbf{N}^+$. Then Equation \eqref{eq:yyycubes} can be rewritten as
\begin{displaymath}
\binom{m+1}{2}=n^3.
\end{displaymath}
Setting $m=1$ we obtain the solution $(x,y)=(3,2)$. Then, it is claimed that every triangular number greater than $1$ is not a cube. 

If $m$ is even, say $m=2k$ for some $k \in \mathbf{N}^+$, then $k(2k+1)=n^3$ and $\mathrm{gcd}(k,2k+1)=1$ so that $k$ and $2k+1$ need to be both cubes. Otherwise $m$ is odd, say $m=2k-1$ for some $k \in \mathbf{N}^+$, then $k(2k-1)=n^3$ and $\mathrm{gcd}(k,2k-1)=1$ so that $k$ and $2k-1$ need to be (coprime) cubes.

Therefore, in both cases we have to find $k$ and $\varepsilon \in \{-1,1\}$ such that $k=a^3$ and $2k+\varepsilon=b^3$ for some coprime $a,b \in \mathbf{N}^+$. By construction, it follows that $$b^3+(-\varepsilon)^3=b^3-\varepsilon=2k=2a^3,$$ hence it is sufficient to show that the Diophantine equation
\begin{equation}\label{euler}
\alpha^3+\beta^3=2\gamma^3 
\end{equation}
admits no non-trivial solutions in $\mathbf{Z}$ (a step-by-step solution of a slightly more general equation than \eqref{euler} can be found in \cite{Sierp}). If $(\alpha,\beta,\gamma)$ is a solution of \eqref{euler} and there exists a prime $r$ dividing both $\alpha$ and $\beta$ then $r$ divides also $\gamma$ and $(\alpha/r,\beta/r,\gamma/r)$ still satisfies \eqref{euler}. That is why we can assume without loss of generality that $\mathrm{gcd}(\alpha,\beta)=1$ and by symmetry $\alpha \le \beta$.

In particular, $\alpha$ and $\beta$ have to be odd, hence we can define coprime integers $u,v \in \mathbf{N}$ such that
\begin{displaymath}
u=\frac{\alpha+\beta}{2}\,\,\,\,\,\text{ and }\,\,\,\,\,v=\frac{\alpha-\beta}{2}.
\end{displaymath}
Therefore, by construction, we get $u(u^2+3v^2)=\gamma^3$, and it is claimed that there are no solutions whenever $v$ is positive.

\begin{list}{$\circ$}{}
\item \textsc{Case $3$ does not divide $u$:} If $u$ is not divisible by $3$ then $\mathrm{gcd}(u,u^2+3v^2)=1$, so there exist coprime integers $c,d\in \mathbf{N}$ such that $u=c^3$ and $u^2+3v^2=d^3$. Defining $t=d-c^2$, we obtain $t(t^2+3tc^2+3c^4)=3v^2$. 

Checking the remainders modulo $3$, we deduce that $t$ is divisible by $3$, $v$ is divisible by $3$, and $t$ is then divisible by $9$. Substituting $t=9e$ and $v=3f$, we can rewrite the equation as $e(27e^2+9e c^2+c^4)=f^2$. At this point, it is straightforward to verify that $\mathrm{gcd}(e,27e^2+9e c^2+c^4)=1$, hence they have to be both squares. It means that we end to solve in $\mathbf{Z}$ an equation in the form 
\begin{equation}\label{form}
\mathfrak{a}^4+9\mathfrak{a}^2\mathfrak{b}^2+27\mathfrak{b}^4=\mathfrak{c}^2.
\end{equation}
\item \textsc{Case $3$ divides $u$:} If $u$ is divisible by $3$, with a similar reasoning, we obtain that there exist $w,z \in \mathbf{N}$ such that $u=9w$ and $v=3z$, so that $w(27w^2+v^2)$ is a cube. Moreover $\mathrm{gcd}(w,27w^2+v^2)=1$, therefore they have to be both cubes, namely $w=\chi^3$ and $27w^2+v^2=\delta^3$ for some $\chi, \delta \in \mathbf{N}$. Define $\nu=\delta-3\chi^2$; then $\nu(\nu^2+9\chi^2\nu+27\chi^4)$ is a square. Again, these factors are coprime and we obtain an equation of the type \eqref{form}.
\end{list}

It means that it is enough to show that Equation \eqref{form} has no solutions in non-zero integers. Without loss of generality we can assume that  $\mathrm{gcd}(\mathfrak{a},\mathfrak{b})=1$. At this point, if $\mathfrak{a}$ is even then $4$ divides $27\mathfrak{b}^4-\mathfrak{c}^2$. Looking at this divisibility modulo $4$, it implies that also $\mathfrak{b}$ is even, contradicting our coprimality assumption $\mathrm{gcd}(\mathfrak{a},\mathfrak{b})=1$. In addition, if $\mathfrak{a}$ and $\mathfrak{b}$ are both odd then, by reasoning modulo $8$, it is easily seen that Equation \eqref{form} has no integral solutions. 

It follows that if $(\mathfrak{a},\mathfrak{b},\mathfrak{c})$ is a solution of Equation \eqref{form} then $\mathfrak{a}$ is odd and $\mathfrak{b}$ is even. In particular, we can define the integer $\mathfrak{m}=\mathfrak{b}/2$. Note that $\mathfrak{a}$ is not divisible by $3$, otherwise $\mathfrak{b}$ would be divisible by $3$ too, looking at the equation modulo $81$. Substituting we can rewrite the equation as
\begin{equation}\label{xyz}
27\mathfrak{m}^4=\left(\frac{\mathfrak{c}+\mathfrak{a}^2}{2}+9\mathfrak{m}^2\right)\left(\frac{\mathfrak{c}-\mathfrak{a}^2}{2}-9\mathfrak{m}^2\right).
\end{equation}
These factors are coprime and positive since their sum and product are both positive. 

Only two cases are possible: in the first case, the factors are $27a^4$ and $b^4$, respectively, for some coprime integers $a,b \in \mathbf{N}$. It implies $3$ divides $27a^4-18\mathfrak{m}^2=b^4+\mathfrak{a}^2$, which is impossible since $-1$ is not a quadratic residue modulo $3$. 

In the second case, the factors are $a^4$ and $27b^4$, respectively, for some coprime integers $a,b \in \mathbf{N}$. It implies that $a^4-18\mathfrak{m}^2=27b^4+\mathfrak{a}^2$, with $ab=\mathfrak{m}$. The variable $a$ cannot be even, looking at the equation modulo $8$. Moreover $a$ and $b$ cannot be both odd since $\mathfrak{a}$ is odd, hence by force $b$ is even. To sum up, we can rewrite Equation \eqref{xyz} as
\begin{equation}\label{xyzfinal}
27b^4=\left(\frac{a^2+\mathfrak{a}}{2}-\frac{9}{2}b^2\right)\left(\frac{a^2-\mathfrak{a}}{2}-\frac{9}{2}b^2\right)
\end{equation}
with $a,b,\mathfrak{a}$ integers such that $a$ and $\mathfrak{a}$ are odd and $b$ is even. 

Similarly, it is not difficult to check that the factors in Equation \eqref{xyzfinal} are coprime and strictly positive, implying that they are, in some order, in the form $27c^4$ and $d^4$ for some $c,d \in \mathbf{N}$. This means that $c^4+9c^2d^2+27c^4=a^2$, which is again in the form of Equation \eqref{form}. 

Notice that $a\le \mathfrak{m} < \mathfrak{b} < \mathfrak{c}$ whenever $\mathfrak{c}\ge 1$. It implies that, if $(\mathfrak{a}^\star,\mathfrak{b}^\star,\mathfrak{c}^\star)$ is a solution of \eqref{form} which minimizes $\mathfrak{c}$ with $\mathfrak{c}\ge 1$, then we can construct another solution $(\mathfrak{a}^{\prime},\mathfrak{b}^{\prime},\mathfrak{c}^{\prime})$ such that $\mathfrak{c}^{\prime}$ is smaller than $\mathfrak{c}^\star$. Therefore the unique solution of Equation \eqref{form} is $(0,0,0)$. 

\subsection{Case $q \ge 5$.} We have to solve the equation
\begin{equation}\label{pevenq5}
\textstyle x^2-y^q=1
\end{equation}    
in integers $\ge 2$, where $q$ is a prime $\ge 5$.

If $x$ is even then $\mathrm{gcd}(x+1,x-1)=1$ and $(x+1)(x-1)=y^q$, implying that $x+1$ and $x-1$ are two coprime $q$-powers that differ by $2$, which is impossible. Hence $x$ has to be odd and $y$ even. 

The following elementary result, which we are not going to prove here, belongs to the folkore and it is commonly known as \emph{Lifting the Exponent}. Typically it is attributed to \'{E}. Lucas \cite{Lucas} and R.D. Carmichael \cite{Carm}, the latter having fixed an error in Lucas' original work in the $2$-adic case.
\begin{lemma}\label{lemma1}
For all integers $a,b,n$ and primes $p$ such that
$n$ is positive, $p$ does not divide $ab$ and $p$ divides $a-b$, the following ones hold:
\begin{list}{$\circ$}{}
\item If $p \ge 3$, then $\upsilon_p(a^n - b^n) =
\upsilon_p(a - b) + \upsilon_p(n)$;
\item If $\upsilon_2(a-b)\ge 2$, then $\upsilon_2(a^n - b^n) =
\upsilon_2(a - b)+\upsilon_2(n)$;
\item If $p = 2$ and $n$ is even, then $\upsilon_2(a^n - b^n) =
\upsilon_2(a - b)+\upsilon_2(a + b) + \upsilon_2(n)-1$.
\end{list}
\end{lemma}

Based on this result, which will be useful also for other cases, we prove the following:
\begin{lemma}\label{lemma2}
Let $p,q$ be primes which are not both equal to $2$ and $a,b$ be distinct coprime integers such that $a^{p}-b^{p}$ is a $q$-th power not divisible by $p$. Then $a-b$ is a $q$-th power.
\end{lemma}
\begin{proof}
By hypothesis, there exists an integer $c$ not divisible by $p$ such that 
\begin{equation}\label{eq:factors}
(a-b)\left(\frac{a^p-b^p}{a-b}\right)=c^q.
\end{equation}
Then, let us suppose that there exists a prime $r$ such that
$$
r \,\,\text{ divides }\,\, \mathrm{gcd}\left(a-b, \frac{a^p-b^p}{a-b}\right).
$$
According to Euler-Fermat's theorem, it holds $a^p-b^p \equiv a-b \bmod{p}$. Since $p$ does not divide $a^p-b^p$, then $r\neq p$. Hence, we obtain by Lemma \ref{lemma1} that, if $r$ divides $a-b$, then
\begin{displaymath}
\upsilon_r\left(\frac{a^p-b^p}{a-b}\right)=\upsilon_r(p)=0.
\end{displaymath}
This implies that the factors in Equation \eqref{eq:factors} are coprime. 

If $q \ge 3$, then we are done because each of them has to be a $q$-th power. Otherwise, we are just left to check the case $q=2$ and $p\ge 3$, together with $a-b=-d^2$, for some non-zero integer $d$. But this is impossible since it would imply that
\begin{displaymath}
\frac{a^p-b^p}{a-b}=\frac{c^2}{-d^2} < 0,
\end{displaymath}
and in particular $a^p-b^p$ and $a-b$ should have different signs.
\end{proof}

At this point, since $\mathrm{gcd}(x+1,x-1)=2$, we can define integers $\varepsilon \in \{-1,1\}$ and coprime $a,b \in \mathbf{N}^+$ such that $x\equiv \varepsilon\pmod{4}$, $y=2ab$, $x+\varepsilon=2a^q$, and $x-\varepsilon =2^{q-1}b^q$. 

Since $q \ge 5$ and $x\ge 2$, we deduce
\begin{displaymath} 
\left(\frac{a}{b}\right)^q = 2^{q-2} \frac{x+\varepsilon}{x-\varepsilon} \ge 8\text{ }\frac{x-1}{x+1}\ge 2, 
\end{displaymath}
implying, in particular, that $a$ is greater than $b$. Moreover, we have by construction
\begin{displaymath} 
a^{2q}-(2\varepsilon b)^q=\left(\frac{x+\varepsilon}{2}\right)^2-2\varepsilon (x-\varepsilon)= \left(\frac{x-3\varepsilon}{2}\right)^2.
\end{displaymath}

Then, according to Lemma \ref{lemma2}, if $q$ does not divide $\frac{1}{2}(x-3\varepsilon)$ then $a^2-2\varepsilon b$ has to be a square: this is not possible since $a^2 \neq a^2-2\varepsilon b$ and $|2\varepsilon b|=2b \le 2(a-1)$, with the consequence that
\begin{displaymath}
(a-1)^2<a^2-2\varepsilon b<(a+1)^2. 
\end{displaymath}
It follows that $q$ divides $\frac{1}{2}(x-3\varepsilon)$ and in particular $q$ does not divide $x$ since $q \ge 5$. Again by Lemma \ref{lemma2}, rewriting Equation \eqref{pevenq5} as $x^2=y^q-(-1)^q$, there exists an integer $\zeta\ge 1$ such that $y-(-1)=\zeta^2$. In particular $\zeta$ is a odd integer and $y$ is not a square, since by assumption $y\ge 2$. It means that $(\zeta,1)$ and $(x,y^{\frac{1}{2}(q-1)})$ are two solutions of the Pell-equation $\alpha^2-y\beta^2=1$. Since $(\zeta,1)$ is its fundamental solution (see for example \cite{Pell} for the theory underlying Pell-equations), there exists a positive integer $m$ such that 
\begin{equation}\label{pell}
x+y^{\frac{q-1}{2}}\sqrt{y}=(\zeta+\sqrt{y})^m.
\end{equation}

Then, in $\mathbf{Z}/y\mathbf{Z}[\sqrt{y}]$ we obtain $x=\zeta^m+m\zeta^{m-1}\sqrt{y}$, which implies that $y$ divides $m\zeta^{m-1}$. In particular, $y$ is even, $\zeta$ is odd, and $m$ is even. 

Looking finally at Equation \eqref{pell} in $\mathbf{Z}/\zeta\mathbf{Z}[\sqrt{y}]$, we obtain $x+y^{(q-1)/2}\sqrt{y}=y^{m/2}$, so that $\zeta$ divides $y^{(q-1)/2}$. Supposing that $\zeta \ge 2$, there exists a prime $r$ such that $r$ divides $\zeta$, which in turn divides $y$, and by construction $r$ divides $\zeta^2-y=-1$, which is impossibile. Therefore we have shown that, given a prime $q\ge 5$, if $x,y,\zeta,m$ verify Equation \eqref{pell} then $\zeta=1$, with the consequence that Equation \eqref{pevenq5} has no solutions.


\section{Proof of Theorem \ref{th:main}.\ref{item:(iii)}: $x$ is a power of $2$.}\label{sec:pow2}

Assuming that $q$ is an odd prime by case \ref{item:(i)}, it is sufficient to rewrite the Diophantine equation \eqref{eq:main} as
\begin{displaymath}
x^p=(1+y)(\underbrace{1-y+y^2-\cdots+y^{q-1}}_{\mathrm{odd }\,\ge 3}),
\end{displaymath}
which does not have clearly any solution.


\section{Proof of Theorem \ref{th:main}.\ref{item:(iv)}: $x \equiv 3,5$ or $7\pmod{8}$.}\label{sec:wfvd}

According to the cases \ref{item:(i)} and \ref{item:(ii)}, we can assume that $p$ and $q$ are distinct odd primes. Note that, since $x$ has remainder $3$, $5$, or $7$ modulo $8$, then $x$ is odd. Hence, $y$ has to be even. It follows that $8$ divides $x^p-1$, which is impossible since
\begin{displaymath}
x^p-1=x \underbrace{\left(x^{\frac{p-1}{2}}+1\right)\left(x^{\frac{p-1}{2}}-1\right)}_{\text{two consecutive even numbers}}+(x-1) \equiv x-1 \pmod{8}
\end{displaymath}
and $8$ does not divide $x-1$ by hypothesis.


\section{Proof of Theorem \ref{th:main}.\ref{item:(v)}: $x$ divides $q$.}\label{sec:xq}

Notice that if $x^p-y^q=1$ for some positive integers $x,y,p,q$ greater than $1$ and $x$ divides $q$ then 
\begin{displaymath}
x^p-Y^x=1,
\end{displaymath}
where $Y:=y^{q/x}$. Accordingly, we can assume without loss of generality that $x=q$ and solve the Diophantine equation $x^p-y^x=1$. By the case \ref{item:(ii)} we know that $x$ has to be odd, hence $y$ is even.

Since $y+1$ divides $y^x+1 = x^p$ then there exists an odd prime $r$ dividing both $x$ and $y+1$. It follows by Lemma \ref{lemma1} that 
\begin{displaymath}
\upsilon_r(x^p)=\upsilon_r\left((y+1)\cdot \frac{y^x+1}{y+1}\right)=\upsilon_r(y+1)+\upsilon_r(x),
\end{displaymath}
so that $(p-1)\upsilon_r(x)=\upsilon_r(y+1)$. In particular, we obtain
\begin{displaymath}
y+1 \ge r^{\upsilon_r(y+1)}= r^{(p-1)\upsilon_r(x)} \ge 3^{p-1} \ge 2^p+1, 
\end{displaymath}
where the last inequality holds for all $p\ge 3$. Therefore $y$ is greater than or equal to $2^p$, with the consequence that
\begin{displaymath}
x^p=y^x+1\ge 2^{px}+1>2^{px}.
\end{displaymath}
We conclude that $x$ is greater than $2^x$ for some integer $x\ge 2$, which is impossible.


\section{Proof of Theorem \ref{th:main}.\ref{item:(vi)}: $x\equiv 1\pmod{y}$.}\label{sec:sec:yxminus}

In this case, we are assuming that there exists $z \in \mathbf{N}^+$ such that $x=yz+1$. Again, we can assume that $p$ and $q$ are distinct odd primes in Equation \eqref{eq:main}, which in turn can be rewritten as $y^q=(yz+1)^p-1$, i.e., 
\begin{displaymath}
y^{q-1}=z \left(1+(yz+1)+\cdots+(yz+1)^{p-1}\right).
\end{displaymath}
Since the second factor has remainder $p$ modulo $z$, then 
\begin{equation}\label{eq:gcd1}
\mathrm{gcd}\left(z,1+(yz+1)+\cdots+(yz+1)^{p-1}\right)\,\,\,\text{ divides }\,\,\,p.
\end{equation} 

Let us first assume that the greatest common divisor in \eqref{eq:gcd1} is $1$. Then there exists coprime $\alpha,\beta \in \mathbf{N}$ such that
\begin{displaymath}
z=\alpha^{q-1} \,\,\,\, \text{ and } \,\,\,\,1+(yz+1)+\cdots+(yz+1)^{p-1}=\beta^{q-1}.
\end{displaymath}
Since $\alpha \beta=y$ by construction, then $yz=\alpha^q \beta$, which implies in turn that
\begin{displaymath}
\sum_{j=0}^{p-1}\left(1+\alpha^q\beta\right)^j=\beta^{q-1}.
\end{displaymath}
At this point, $\beta$ is an integer greater than $1$ and the right hand side has remainder $p$ modulo $\beta$. Therefore $\beta$ has to be exactly $p$, with the consequence that
\begin{displaymath}
p^{q-1}=\sum_{j=0}^{p-1}\left(1+\alpha^qp\right)^j=\sum_{j=0}^{p-1}\sum_{k=0}^j\binom{j}{k}\alpha^{qk}p^k=p+\sum_{j=1}^{p-1}\sum_{k=1}^j\binom{j}{k}\alpha^{qk}p^k.
\end{displaymath}
One the one hand, it is easily seen that
\begin{displaymath}
\begin{split}
\sum_{j=1}^{p-1}\sum_{k=1}^j\binom{j}{k}\alpha^{qk}p^k&=\left(1+2+\cdots+(p-1)\right)\alpha^qp+\sum_{j=2}^{p-1}\sum_{k=2}^j\binom{j}{k}\alpha^{qk}p^k \\ 
&=p^2 \frac{p-1}{2}\alpha^q+p^2 \sum_{j=2}^{p-1}\sum_{k=2}^j\binom{j}{k}\alpha^{qk}p^{k-2} \equiv 0\pmod{p^2}.
\end{split}
\end{displaymath}
On the other hand, since $q$ is an odd prime then $p^2$ divides $p^{q-1}$, the above sum cannot be a multiple of $p^2$.

Let us assume now that the greatest common divisor in \eqref{eq:gcd1} is $p$. Then there exist $u,v \in \mathbf{N}$ such that 
\begin{displaymath}
z=pu, \,\,\,\, \,\,\,\,1+(yz+1)+\cdots+(yz+1)^{p-1}=pv,\,\,\,\, \text{ and } \,\,\,\,y^{q-1}=p^2uv.
\end{displaymath}
Hence, $y$ is a divisible by $p$, and there exist positive integers $h,k$ such that $y=p^hk$ such that $p$ does not divide $k$. As in the previous case, we can see that $p$ does not divide $v$, indeed
\begin{displaymath}
\sum_{j=0}^{p-1}(yz+1)^j=\sum_{j=0}^{p-1}\left(\underbrace{p^{h+1}ku}_{\text{divisible by }p^2}+1\right)^j \equiv p\pmod{p^2}.
\end{displaymath}
It implies that there exist coprime positive integers $\gamma,\delta$, not divisible by $p$, such that 
\begin{displaymath}
z=p^{h(q-1)-1}\gamma^{q-1}, \,\,\,\, \,\,\,\,1+(yz+1)+\cdots+(yz+1)^{p-1}=p\delta^{q-1},\,\,\,\, \text{ and } \,\,\,\,y=p^h\gamma\delta.
\end{displaymath}
Accordingly, we deduce that
\begin{displaymath}
p\delta^{q-1}=\sum_{j=0}^{p-1}(yz+1)^j=\sum_{j=0}^{p-1}\left(p^{hq-1}\gamma^q\delta+1\right)^j.
\end{displaymath}
Multiplying both sides by $p^{hq-1}\gamma^q\delta$ we get
\begin{equation}\label{eq:final2}
y^q=\left(p^{hq-1}\gamma^q\delta+1\right)^p-1=\sum_{j=1}^{p}\binom{p}{j}p^{j(hq-1)}\gamma^{qj}\delta^j.
\end{equation}
If $\delta$ is greater than $1$, then it has to be divisible by some prime $r$, so that in particular
\begin{displaymath}
q\upsilon_r(\delta)=\upsilon_r\left(\sum_{j=1}^{p}\binom{p}{j}p^{j(hq-1)}\gamma^{qj}\delta^j\right)=\upsilon_r\left(p^{hq-1}\gamma^q\delta+\sum_{j=2}^{p}\binom{p}{j}p^{j(hq-1)}\gamma^{qj}\delta^j\right).
\end{displaymath}
Recalling that $\delta$, $\gamma$, and $p$ are (pairwise) coprime, then the right hand side is exactly $\upsilon_r(\delta)$, which is impossible since $q$ is greater than $1$. Therefore $\delta=1$, and the Equation \eqref{eq:final2} simplifies to
\begin{displaymath}
\sum_{j=2}^p{\binom{p}{j}p^{j(hq-1)}\gamma^{jq}}=0,
\end{displaymath}
which cannot hold because it is a (non-empty) sum of positive integers.


\section{Proof of Theorem \ref{th:main}.\ref{item:(vii)}: $y$ is power of a prime.}\label{sec:ypower}

According to the cases \ref{item:(i)} and \ref{item:(ii)} proved in Sections \ref{sec:caseqeven} and \ref{sec:casepeven}, respectively, if $(x,y,p,q)$ is a solution of the Equation \eqref{eq:main} different from $(3,2,2,3)$ then $p$ and $q$ are [distinct] odd primes. At this point, let us rewrite Equation \eqref{eq:main} as
\begin{equation}\label{eq:main2}
y^q=(x-1)\cdot \left(\frac{x^p-1}{x-1}\right).
\end{equation}
The first factor has to be greater than $1$ by case \ref{item:(iii)}. In addition, the second factor (which is greater than $x-1$) has to be odd, indeed
\begin{displaymath}
\frac{x^p-1}{x-1}=1+x+\cdots+x^{p-1} \equiv 1+x(\underbrace{p-1}_{\mathrm{even}}) \equiv 1 \pmod{2}.
\end{displaymath}
In particular, also $x-1$ has to be odd: indeed, supposing that $y=\ell^c$ for some prime $\ell$ and positive integer $c$, then $x-1=\ell^a$ and $\frac{x^p-1}{x-1}=\ell^b$ for some integers $1\le a < b$ such that $a+b=cq$.

Then, according to Lemma \ref{lemma1}, we obtain
\begin{displaymath}
\upsilon_\ell\left(\frac{x^p-1}{x-1}\right)=\upsilon_\ell(p).
\end{displaymath}
It implies that the greatest common divisor between $x-1$ and $\frac{x^p-1}{x-1}$ is exactly $p$, which in turn is equal to $\ell$. Therefore, we can say that 
\begin{displaymath}
x-1=p \,\,\,\,\,\,\,\text{ and }\,\,\,\,\,\,\,\frac{x^p-1}{x-1}=p^b
\end{displaymath}
for some integer $b \ge 2$. We can conclude that
\begin{displaymath}
p^{b+1}=(x-1)p^b=x^p-1=(p+1)^p-1=\sum_{j=1}^p\binom{p}{j}p^j=p^2+p^3\cdot \frac{p-1}{2}+\sum_{j=3}^p \binom{p}{j}p^j,
\end{displaymath}
which is impossible since the left hand side is divisible by $p^3$ while the right hand side is not.


\section{Proof of Theorem \ref{th:main}.\ref{item:(viii)}: $y\le \min\left(\frac{(pq)^p}{q},\left(\frac{q\sqrt{p}}{2}\right)^p\right)$.}\label{sec:y5p}


As before, 
according to the cases \ref{item:(i)} and \ref{item:(ii)}, 
we can assume that $p$ and $q$ are distinct odd primes. Then, reasoning as in the case \ref{item:(vii)}, we rewrite Equation \eqref{eq:main} in the form \eqref{eq:main2}, and notice that for each odd prime $r$ dividing $x-1$ we get by Lemma \ref{lemma1} 
\begin{displaymath}
\upsilon_r\left(\frac{x^p-1}{x-1}\right)=\upsilon_r(p),
\end{displaymath}
which implies in turn that
\begin{equation}\label{eq:gcd}
\mathrm{gcd}\left(x-1,\frac{x^p-1}{x-1}\right) \,\,\, \mathrm{divides} \,\,\,p.
\end{equation}

Let us suppose, at first, that the greatest common divisor in \eqref{eq:gcd} is $1$, i.e., there exist two coprime integers $a,b \ge 2$ with product $y$ such that
\begin{displaymath}
x-1=a^q \,\,\,\,\,\,\text{ and }\,\,\,\,\,\,\frac{x^p-1}{x-1}=b^q.
\end{displaymath}
In addition, let us assume the following Cassell's result \cite{Cassell}, which can be proved by means of elementary methods.
\begin{lemma}\label{lem:cassell}
Fix integers $x,y\ge 2$ and primes $p,q\ge 3$ such that $|x^p-y^q|=1$. Then $p$ divides $y$ and $q$ divides $x$.
\end{lemma}

Lemma \ref{lem:cassell} immediately implies that $p$ divides $y^q=(x-1)\left(\frac{x^p-1}{x-1}\right)$, hence $p$ divides at least one between $x-1$ and $\frac{x^p-1}{x-1}$. If $p$ divides $x-1$ then, thanks to Lemma \ref{lemma1}, $p$ divides also $\frac{x^p-1}{x-1}$. Conversely, if $p$ divides also $\frac{x^p-1}{x-1}$, then $x^p\equiv 1\pmod{p}$ and, by Euler-Fermat's theorem, $x\equiv 1\pmod{p}$ as well. This contradicts the coprimality assumption between $x-1$ and $\frac{x^p-1}{x-1}$.

On the other hand, let us suppose that the greatest common divisor in \eqref{eq:gcd} is $p$, i.e., there exist two coprime integers $c,d \ge 2$ with product $y$ such that
\begin{displaymath}
x-1=p^{q-1}c^q \,\,\,\,\,\,\text{ and }\,\,\,\,\,\,\frac{x^p-1}{x-1}=pd^q.
\end{displaymath}

Recalling that $q$ divides $x$ by Lemma \ref{lem:cassell}, we obtain 
$$
0\equiv x=1+p^{q-1}c^q \equiv 1+1\cdot c^q \equiv 1+c \pmod{q},
$$
so that there exists $X \in \mathbf{N}^+$ such that
$$
x=1+p^{q-1}(Xq-1)^q.
$$

Rewriting Equation \eqref{eq:main} as
$
(y+1)\left(\frac{y^q+1}{y+1}\right)=x^p
$, 
it can be seen with a similar argument that there exists $Y \in \mathbf{N}$ such that
$$
y=-1+q^{p-1}(pY+1)^p.
$$

Let us suppose that $Y \ge 1$. Then the hypothesis $y \le q^{p-1}p^p$ implies
$$
q^{p-1}p^p \ge y\ge -1+q^{p-1}(p+1)^p > q^{p-1}((p+1)^p-1)>q^{p-1}p^p,
$$
which is false. This forces $Y=0$, so that
$
y=q^{p-1}-1.
$

Finally, the hypothesis $y\le p^{p/2}(q/2)^{p}$ implies
\begin{displaymath}
\begin{split}
1+p^{q-1}\left(q-1\right)^q \le x & = (y^q+1)^{1/p} \le \left(\left(p^{\frac{p}{2}}\left(q/2\right)^{p}\right)^q+1\right)^{1/p}\\
&< \left(\left(p^{p \frac{q-1}{q}}\left(q/2\right)^{p}\right)^q+1\right)^{1/p} \\
&< \left(p^{p \frac{q-1}{q}}\left(q-1\right)^{p}\right)^{q/p}= p^{q-1}\left(q-1\right)^{q},
\end{split}
\end{displaymath}
which is false. This completes the proof.


\section{Closing Remarks}\label{sec:final}

There are many open questions related to Catalan's equation. For instance, Pillai \cite{Pillai36, Pillai45} conjectured in 1936 that, for each $k \in \mathbf{N}^+$, the Diophantine equation
$$
x^p-y^q=k
$$
has only finitely many positive integer solutions $(x,y,p,q)$, with $p\ge 2$ and $q\ge 2$. 

It is worth noting that Erd\"{o}s formulated an even stronger conjecture, i.e., 
$$
a_{n+1}-a_n \gg n^\varepsilon
$$
for some $\varepsilon>0$, where $a_n$ stands for the $n$-th element of the increasing sequence of perfect powers, so that $a_1=1$, $a_2=4$, $a_3=8$, $a_4=9$, $\ldots$.

On a similar note about perfect powers, the banker A. Beal offered in 1993 a \$1 million prize for a proof of the following conjecture: Given $x,y,z,p,q,r \in \mathbf{N}^+$ with $p,q,r>2$ such that
$$
x^p+y^q=z^r
$$
then $\mathrm{gcd}(x,y,z)>1$. Actually, the prize remains unclaimed (however, the abc conjecture would imply that there are at most finitely many counterexamples to Beal's conjecture).

A beautiful survey of open problems and history of classical questions related to perfect powers can be found in \cite{Wald}.


\end{document}